\documentclass[11pt]{amsart}
\usepackage{amsmath,amsfonts,amssymb,amscd,verbatim,amsthm,amsgen,xspace,graphicx,color}
\usepackage{srcltx}
\usepackage{enumerate}
\date{\today}

\newcommand{\ka}{\mathfrak{k}}
\newcommand{\p}{\mathfrak{p}}
\newcommand{\g}{\mathfrak{g}}

\newcommand{\C}{{\ensuremath{\mathbb{C}}}}

\newcommand\Ad{\operatorname{Ad}}

\newcommand{\Kt}{\widetilde{K}}

\def\dim{\mathop{\hbox {dim}}\nolimits}
\def\Ad{\mathop{\hbox {Ad}}\nolimits}
\def\ad{\mathop{\hbox {ad}}\nolimits}

\def\im{\mathop{\hbox {Im}}\nolimits}

\def\ker{\mathop{\hbox{Ker}}\nolimits}

\def\Cas{\mathop{\hbox {Cas}}\nolimits}

\newcommand{\ugcp}{U(\frg)\otimes C(\frp)}
\newcommand{\pf}{\begin{proof}}
\newcommand{\epf}{\end{proof}}
\newcommand{\eq}{\begin{equation}}
\newcommand{\eeq}{\end{equation}}
\newcommand{\eqn}{\begin{equation*}}
\newcommand{\eeqn}{\end{equation*}}
\newcommand{\twedge}{\textstyle\bigwedge}

\newcommand{\fra}{\mathfrak{a}}

\newcommand{\frg}{\mathfrak{g}}
\newcommand{\frh}{\mathfrak{h}}
\newcommand{\frk}{\mathfrak{k}}

\newcommand{\frn}{\mathfrak{n}}

\newcommand{\frp}{\mathfrak{p}}

\newcommand{\frt}{\mathfrak{t}}

\newcommand{\frsl}{\mathfrak{sl}}

\newcommand{\frso}{\mathfrak{so}}

\newcommand{\bbC}{\mathbb{C}}

\newcommand{\bbZ}{\mathbb{Z}}

\newcommand{\caH}{\mathcal{H}}

\newcommand{\tr}{\operatorname{tr}}

\newcommand{\cspan }{\operatorname{span}}

\theoremstyle{plain}
\newtheorem{theorem}{Theorem}[section]
\newtheorem{cor}[theorem]{Corollary}
\newtheorem{prop}[theorem]{Proposition}
\newtheorem{lemma}[theorem]{Lemma}
\newtheorem{remark}[theorem]{Remark}

\date{\today}
\title{K-invariants in the algebra $U(\mathfrak{g}) \otimes C(\mathfrak{p})$ for the group $SU(2,1)$}
\author{Ana Prli\' c}\thanks{This work was partially supported by a grant from the Croatian Science Foundation.}



\address{Department of Mathematics, University of Zagreb,
Bijeni\v cka cesta 30, 10000 Zagreb, Croatia.}
\email{anaprlic@math.hr}

\begin{document}

\begin{abstract}
Let $\g = \ka \oplus \p$ be the Cartan decomposition of the complexified Lie algebra $\g=\mathfrak{sl}(3,\bbC)$ of the group $G=SU(2,1)$. Let 
$K=S(U(2)\times U(1))$; so $K$ is a maximal compact subgroup of $G$. Let $U(\frg)$ be the universal enveloping algebra of $\frg$, and let $C(\frp)$ be the Clifford algebra with respect to the trace form $B(X,Y)=\tr(XY)$ on $\frp$. 
We are going to prove that the algebra of $K$--invariants in $U(\g) \otimes C(\p)$ is generated by five explicitly given elements. This is useful for studying algebraic Dirac induction for $(\frg,K)$-modules. Along the way we will also recover the (well known) structure of the algebra $U(\frg)^K$.
\end{abstract}

\maketitle

\section{Introduction}
Let $G$ be a connected real reductive Lie group with the Cartan involution $\Theta$, such that $K = G^{\Theta}$ is a maximal compact subgroup of $G$. Let $\g = \ka \oplus \p$ be the corresponding Cartan decomposition of the complexified Lie algebra of $G$.

A well known theorem due to Harish-Chandra \cite{HC} asserts that an irreducibile $(\g, K)$--module is characterized by the action of $U(\g)^{K}$ on any non-trivial $K$--isotypic component. Here $U(\g)$ denotes the universal enveloping algebra of $\g$. A simplified algebraic proof of this result was given by Lepowsky-McCollum \cite{LMC}.

The following version of that theorem was proved in \cite{PR}. Let $X$ be a $(\g, K)$--module. Let $S$ be the spin module for the Clifford algebra $C(\frp)$ of $\p$ with respect to the trace form $B(X,Y)=\tr(XY)$. Let $\tilde{K}$ be the spin double cover of $K$. Then the action of $K$--invariants in $U(\g) \otimes C(\p)$ on any nontrivial $\tilde{K}$-isotypic component of $X \otimes S$ determines an irreducible $(U(\g) \otimes C(\p), \tilde{K})$--module $X \otimes S$ up to isomorphism. 

The modules of the form $X\otimes S$ are important in the setting of Dirac operator actions and Dirac cohomology.
Let $D\in U(\frg)\otimes C(\frp)$ be the Dirac operator, defined as follows (\cite{P1}; \cite{V}). Let $b_i$ be any basis of $\frp$ and let $d_i$ be the dual basis with respect to $B$. Then
\[
D=\sum_i b_i\otimes d_i.
\]
It is easy to see that $D$ is independent of the choice of $b_i$, and $K$-invariant for the action $\Ad\otimes\Ad$ of $K$ on $\ugcp$. One of its main properties is the following formula for $D^2$ due to Parthasarathy \cite{P1}:
\eqn
D^2=-(\Cas_\frg\otimes 1+\|\rho_\frg\|^2)+(\Cas_{\frk_\Delta}+\|\rho_\frk\|^2).
\eeqn
Here $\Cas_\frg$ is the Casimir element of $U(\frg)$ and $\Cas_{\frk_\Delta}$ is the Casimir element of $U(\frk_\Delta)$,
where $\frk_\Delta$ is the diagonal copy of $\frk$ in $U(\frg)\otimes C(\frp)$, defined using the obvious embedding
$\frk\hookrightarrow U(\frg)$ and the usual map $\frk\to\frso(\frp)\to C(\frp)$. See \cite{HP2} for details.

If $X$ is a $(\frg,K)$-module, and if $S$ is a spin module for $C(\frp)$, then $D$ acts on $X\otimes S$. The Dirac cohomology of $X$ is the $\Kt$-module
\[
H_D(X)=\ker D / \im D\cap\ker D.
\]
If $X$ is admissible, then $H_D(X)$ is finite-dimensional. This follows from the above formula for $D^2$, which implies that $\ker D^2$ is finite-dimensional, and from the obvious fact that $H_D(X)$ is the cohomology of the differential $D\big|_{\ker D^2}$. 

If $X$ is unitary, then
\eq
\label{HDunit}
H_D(X)=\ker D=\ker D^2.
\eeq
This follows from the existence of a natural inner product on $X\otimes S$, such that $D$ is self-adjoint with respect to this inner product. This also implies Parthasarathy's Dirac inequality, $D^2\geq 0$ \cite{P2}. Written more explicitly using the formula for $D^2$, this becomes a powerful necessary condition for unitarity. The situation is similar for a finite-dimensional module $X$; (\ref{HDunit}) still holds, and $D^2\leq 0$.

The main result about Dirac cohomology is the following theorem. It was conjectured by Vogan
\cite{V}, and proved by Huang and Pand\v zi\'c \cite{HP1}. 

Let $\frh=\frt\oplus\fra$ be a fundamental Cartan
subalgebra of $\frg$. We view $\frt^*\subset\frh^*$ by extending
functionals on $\frt$ by 0 over $\fra$. Denote by $R_{\frg}$ (resp. $R_{\frk}$)
the set of $(\frg,\frh)$-roots (resp. $(\frk,\frt)$-roots). We fix compatible 
positive root systems $R^{+}_{\frg}$ and  $R^{+}_{\frk}$ for $R_\frg$ and $R_\frk$ respectively.
In particular, this determines the half-sums of positive roots $\rho_\frg$ and $\rho_\frk$. Write $W_{\frg}$ (resp. $W_{\frk}$) for the
Weyl group associated with $(\frg,\frh)$-roots
(resp. $(\frk,\frt)$-roots). 

\begin{theorem}
\label{HPmain}
Let $X$ be a $(\frg,K)$-module with infinitesimal character corresponding to $\Lambda\in\frh^*$ via the Harish-Chandra isomorphism. Assume
that $H_D(X)$ contains the irreducible $\Kt$-module $E_\gamma$ with highest weight $\gamma\in\frt^*$. 
Then $\Lambda$ is equal to $\gamma+\rho_\frk$ up to conjugation by the Weyl group $W_\frg$. In other words, the $\frk$-infinitesimal character of any $\Kt$-type contributing to $H_D(X)$ 
is $W_\frg$-conjugate to the $\frg$-infinitesimal character of $X$.
\end{theorem}

By now, Dirac cohomology has been calculated for many (unitary) modules, see \cite{HKP}, \cite{HPP}, \cite{BP1}, \cite{BP2}. It has been related to other kinds of cohomology of $(\frg,K)$-modules, like $\frn$-cohomology \cite{HPR} and $(\frg,K)$-cohomology \cite{HP1}, \cite{HKP}. It has also been related to some classical topics in representation theory, like (generalized) Weyl character formula and Bott-Borel-Weil Theorem, the construction of discrete series, and multiplicities of automorphic forms \cite{HP2}. It can also be successfully applied to some classical branching problems  \cite{HPZ}.  
The definition and some of the results, notably Theorem \ref{HPmain}, have been extended to several other settings \cite{AM}, \cite{Ku}, \cite{Ko2}, \cite{HP3}, \cite{KMP}, \cite{BCT}.

Understanding the $(U(\g) \otimes C(\p))^{K}$--action on the Dirac cohomology is important for studying the algebraic Dirac induction developed in \cite{PR}. The goal of algebraic Dirac induction is to construct $(\g, K)$--modules $X$ whose Dirac cohomology is (or contains) a given irreducible $\tilde{K}$--module $W$. Pand\v zi\'c and Renard give two main constructions, which satisfy certain adjunction properties with respect to (mild modifications of) Dirac cohomology. Each of the constructions has several versions, depending on how much extra structure (coming from
the $(U(\g) \otimes C(\p))^{K}$--action) one wants to put on the Dirac cohomology. On the one hand, it would be good to have all of 
$(U(\g) \otimes C(\p))^{K}$ acting, but the problem is that it is expected to be hard to study the structure and module theory of $(U(\g) \otimes C(\p))^{K}$. (Recall that in general it is very hard to study the structure and module theory of the algebra $U(\g)^{K}$, which is contained in $(U(\g) \otimes C(\p))^{K}$.)

However, for the case $G = SU(2, 1)$ the situation is much simpler than in general, and we prove that the algebra $(U(\g) \otimes C(\p))^{K}$ is generated by five elements. Two of them are in the center $Z(\frk)$ of $U(\ka)$ -- the Casimir element and the element spanning the center of $\frk$.
One of the generators is in another abelian algebra, $C(\p)^{K}$ (which is three-dimensional in this case). The fourth generator is the Dirac operator, and the fifth generator is another distinguished element that can be thought of as a $\ka$-version of the Dirac operator. 

All the above generators are sufficiently explicit so that their action on Dirac cohomology of many modules can be calculated explicitly.
This result is important for understanding the algebraic Dirac induction for the nonholomorphic discrete series of the group $SU(2,1)$ \cite{Pr}.

The strategy we use to study the algebra $(U(\frg)\otimes C(\frp))^K$ is to first study the $K$--invariants in the tensor product of the symmetric algebra $S(\g)$ of $\g$ and the exterior algebra $\twedge  (\p)$ of $\p$. Namely, the algebras $U(\g) \otimes C(\p)$ and $S(\g) \otimes \twedge  (\p)$ are isomorphic as $K$--modules, and the algebra structure of $S(\g) \otimes \twedge  (\p)$ is much simpler.

The paper is organized as follows. In Section $2$ we describe the $K$-module structure of $S(\ka)$ and $S(\p)$. These are special cases of results of Kostant
\cite{K} and Kostant-Rallis \cite{KR}, but our description is elementary and completely explicit. Then we calculate the dimension of the space of invariants in $S(\g) \otimes \twedge  (\p)$ for each degree. In Section $3$ we give a basis for the vector space $(S(\g) \otimes \twedge  (\p))^{K}$. Finally, in Section $4$ we prove the main result, that the algebra $(U(\g) \otimes C(\p))^{K}$ is generated by the five elements mentioned above.

Among the side results we obtain, let us mention the well known fact $U(g)^K=Z(\frg)Z(\frk)\cong Z(\frg)\otimes Z(\frk)$ \cite{J}. 
In particular, $U(\frg)^K$ is abelian and this explains why all irreducible $(\frg,K)$-modules have only multiplicity one
$K$-types. (Namely, it is part of the above mentioned results of \cite{HC} and \cite{LMC} that $U(\frg)^K$ acts irreducibly on the multiplicity space of each $K$-type of an irreducible $(\frg,K)$-module). For a general pair $(\frg,K)$ with $\frg$ simple noncompact and $K$ connected, $U(\frg)^K=U(\frg)^\frk$ is not abelian, but its center is always $Z(\frg)Z(\frk)\cong Z(\frg)\otimes Z(\frk)$, as was shown by Knop in \cite{Kn}.

We also prove that $(U(\g) \otimes C(\p))^{K}$ is a free module over $U(\frg)^K$, of rank $16=\dim C(\frp)$. We hope to be able to generalize this result in future.

I would like to thank Pavle Pand\v{z}i\'{c} and Hrvoje Kraljevi\'{c} for all suggestions and ideas they shared with me. 

\vspace{.2in}

\section{Degrees of invariants in $S(\g) \otimes \twedge  (\p)$}

We will denote by $G$ the Lie group 
\[
SU(2,1) = \{ g \in SL(3, \mathbb{C}) \, | \, g^* \gamma g = \gamma \},
\]
where $\gamma = \begin{pmatrix}1 & 0 & 0\cr 0 & 1 & 0\cr 0 & 0& -1\end{pmatrix}$. 
The (real) Lie algebra of $G$ is 
\[
\g_0 = \mathfrak{su}(2,1)= \{ x \in \mathfrak{sl}(3, \mathbb{C}) \, | \, x^* = - \gamma x \gamma\}.
\] 
The complexification of $\frg_0$ is $\g=\frsl(3,\bbC)$. One basis for $\g$ is given by:
\begin{gather*}
H_1 = \frac{1}{3}(2e_{11}-e_{22} - e_{33}),\quad H_2  = \frac{1}{3}(-e_{11}+2e_{22} - e_{33}), \\ 
E  = e_{12}, \quad F  = e_{21},\quad E_1  = e_{13}, \quad  E_2  = e_{23}, \quad  F_1  = e_{31}, \quad  F_2  = e_{32},
\end{gather*}
where $e_{ij}$ denotes the usual matrix unit: it has the $ij$ entry equal to 1 and all other entries equal to 0. The elements $H_1$ and $H_2$ do not look the simplest possible, but they fit well with the subsequent computations.
The commutation relations are given by
\begin{align}\label{tcom}
&[H_1, E_1] = E_1, \quad &[H_2, E_1] = 0, \quad &[H_1, E_2] = 0,  \quad &[H_2, E_2] = E_2 \\
&[H_1, F_1] = - F_1, \quad &[H_2, F_1] = 0, \quad &[H_1, F_2] = 0, &[H_2, F_2]  = -F_2 \notag \\
&[H_1, E] = E , \quad &[H_2, E] = - E , \quad &[H_1, F] = - F , \quad &[H_2, F] = F \notag,
\end{align}
Let $\g = \ka \oplus \p$ be the Cartan decomposition of $\g$ corresponding to the usual Cartan involution $\theta (X) = - X^{*}$. Then
\[
\ka = \cspan  \{ H_1, H_2, E_, F \} \cong \mathfrak{gl}(2, \mathbb{C}),\qquad \text{ and }\quad \p = \cspan  \{ E_1, E_2, F_1, F_2 \}.
\]
We denote the elements $H_1-H_2$ and $H_1+H_2$ of $\frk$ by $H$ respectively $a$. Then the semisimple part of $\frk$ is 
\[
\ka_s = \cspan   \{ H, E, F \} \cong \frsl(2,\bbC),
\]
with $H$, $E$ and $F$ corresponding to the standard basis of $\frsl(2,\bbC)$, while the center of $\frk$ is equal to $\bbC a$. 
We also set
\[
b = H^2 + 4 EF \quad \in S(\ka_s) \subset S(\g).
\]
(Note that $b$ symmetrizes to a multiple of the Casimir element of $U(\ka_s)$.) We will view $a\in\frk$ as an element of $S(\frk)$. Both $a$ and $b$ are easily seen to be $K$-invariant.

\begin{lemma}
\label{lemmaK}
For each integer $n\geq 2$, 
\[
S^{n}(\ka_s) = V_{2n} \oplus b S^{n-2}(\ka_s)
\] 
as a $\frk_s$-module, where $V_{2n}$ is the $\mathfrak{sl}(2, \mathbb{C})$--module with the highest weight $2n$ and a highest weight vector $E^n$. Furthermore, $S^{1}(\ka_s)=\frk_s$ is the module $V_2$ with the highest weight 2 and a highest weight vector $E$, and $S^{0}(\ka_s)$ is the trivial module 
$V_0$ spanned by 1.
\end{lemma}
\begin{proof}
The cases $n=0,1$ are obvious. Let $n\geq 2$. 
It is clear from the commutator table (\ref{tcom}) that $E^{n}$ is a vector of weight $2n$. Furthermore, 
\[
\dim   V_{2n} = 2n+1; \quad \dim b S^{n-2}(\ka_s) = \binom{n}{2}; \quad \dim   S^{n}(\ka_s) = \binom{n + 2}{2}.
\]
It follows that 
$$
\dim  S^{n}(\ka_s) = \dim b S^{n-2}(\ka_s) + \dim   V_{2n}.
$$
On the other hand, all the weights in the $\mathfrak{sl}(2, \C)$--module $b S^{n-2}(\ka_s)$ are strictly smaller than $2n$, so 
$V_{2n} \cap b S^{n-2}(\ka_s) = 0$. Hence the sum is direct
\end{proof}

\begin{remark}
\label{rmkK}
{\rm 
Lemma \ref{lemmaK} implies a special case of Kostant's theorem \cite{K}, which says that for any complex semisimple Lie algebra, the symmetric algebra decomposes into a tensor product of the subalgebra of invariants and the space of harmonics. In our case, the algebra of invariants is clearly $\bbC[b]$, while the space of harmonics is
\eq
\label{hks}
\caH_{\frk_s}=\bigoplus_{n\in\bbZ_+} V_{2n}.
\eeq
If we take this simply as notation (and leave to the interested reader to check that this indeed agrees with Kostant's definition of harmonics), then Lemma \ref{lemmaK} implies that, as a $\frk_s$-module,
\eq
\label{Kks}
S(\frk_s) \cong S(\frk_s)^{\frk_s}\otimes \caH_{\frk_s}=\bbC[b]\otimes \caH_{\frk_s}.
\eeq
}
\end{remark}

Since $\frk=\bbC a\oplus\frk_s$, it follows $S(\frk)=\bbC[a]\otimes S(\frk_s)$, and so 
\eq
\label{Kk}
S(\frk) \cong \bbC[a,b]\otimes \caH_{\frk_s}.
\eeq
The $\frk_s$-action on $\caH_{\frk_s}$ is given by (\ref{hks}), and the action of $a\in\frk$ is trivial. As usual, we will label finite-dimensional $\frk$-modules by their highest weights, which we identify with pairs $(\alpha,\beta)\in\bbC^2$ such that $\alpha-\beta\in\bbZ_+$. A vector $v$ in a 
$\frk$-module is of weight $(\alpha,\beta)$ if $H_1v=\alpha v$ and $H_2v=\beta v$. The finite-dimensional $\frk$-module with highest weight 
$(\alpha,\beta)$ will be denoted by $V_{(\alpha,\beta)}$. 
With this notation we have
\eq
\label{hk}
\caH_{\frk_s}=\bigoplus_{n\in\bbZ_+} V_{(n,-n)}
\eeq
as $\frk$-modules. Namely, the module $V_{(n,-n)}$ is equal to $V_{2n}$ as a $\frk_s$ module, and $a=H_1+H_2$ acts trivially on $V_{(n,-n)}$.

We now turn to analyzing the $K$-structure of $S(\frp)$. It is easy to see that the element
\[
c = E_1 F_1 + E_2 F_2
\]
is $K$-invariant. 

\begin{lemma}
\label{lemmaKR}
Let $V_{(n-i, -i)} \subset S(\p)$ be the $\ka$--module with highest weight $(n-i, -i)$ and highest weight vector $E_{1}^{n-i} F_{2}^{i}$. Then for 
$n\geq 2$ we have
$$
S^{n}(\p) = \left(  V_{(n, 0)} \oplus V_{(n-1, -1)} \oplus \cdots \oplus V_{(0, -n)} \right) \oplus c S^{n-2}(\p).
$$
Furthermore, $S^0(\p)$ is a trivial $\frk$-module spanned by $1$, while $S^1(\frp)=\frp=\frp^+\oplus\frp^-= V_{(1,0)} \oplus V_{(0,-1)}$.
\end{lemma}

\begin{proof}  
Using the commutator table (\ref{tcom}), it is easy to see that 
\[
(\ad H_1)(E_{1}^{n-i} F_{2}^{i}) = (n-i)E_{1}^{n-i} F_{2}^{i}, \quad (\ad    H_2)(E_{1}^{n-i} F_{2}^{i}) = -iE_{1}^{n-i} F_{2}^{i}.
\] 
Since for any $i \in \{ 0, 1, 2, \cdots, n \}$, $V_{(n-i, -i)}$ is an irreducible module for $\frk_s=\mathfrak{sl}(2, \mathbb{C})$, with highest weight $n$, while the highest $\frk_s$-weight in $c S^{n-2}(\p)$ is $n-2$, we conclude
\eq
\label{pf1}
\left(  V_{(n, 0)} \oplus V_{(n-1, -1)} \oplus \cdots \oplus V_{(0, -n)} \right) \cap c S^{n-2}(\p) = 0.
\eeq
Furthermore, $\dim   \left(  V_{(n, 0)} \oplus V_{(n-1, -1)} \oplus \cdots \oplus V_{(0, -n)} \right)  = (n+1)^2$, $\dim  (S^{n}(\p)) = \binom{n + 3}{3}$ and $\text{dim } (cS^{n-2}(\p)) = \binom{n + 1}{3}$, and this implies
\[
\dim   S^{n}(\p) = \dim   (\left(  V_{(n, 0)} \oplus V_{(n-1, -1)} \oplus \cdots \oplus V_{(0, -n)} \right) \oplus c S^{n-2}(\p)).
\]
Together with (\ref{pf1}), this implies the claim for $n\geq 2$. The cases $n=0,1$ are obvious.
\end{proof}

\begin{remark}
\label{rmkKR}
{\rm Lemma \ref{lemmaKR} implies a special case of a theorem of Kostant and Rallis \cite{KR}, which says that $S(\frp)$ can be written as a tensor product of the algebra $S(\frp)^K$ of invariants and the space of harmonics. By Lemma \ref{lemmaKR}, if we define the space of harmonics as
\eq
\label{hp}
\caH_\frp=\bigoplus_{n\in\bbZ_+}\bigoplus_{i=0}^n  V_{(n-i, -i)},
\eeq 
(and one can easily check that this does agree with the definition of \cite{KR}), then we have
\eq
\label{Kp}
S(\frp)=S(\frp)^K\otimes \caH_\frp=\bbC[c]\otimes \caH_\frp.
\eeq
}
\end{remark}

Using (\ref{Kk}) and (\ref{Kp}), we can write
\eq
\label{sglpdec}
S(\g) \otimes \twedge  (\p) \cong \mathbb{C}[a, b, c] \otimes \mathcal{H}_{\ka_s} \otimes \mathcal{H_{\p}} \otimes \twedge(\p).
\eeq
It is not hard to determine the $K$-structure of $\twedge(\frp)$. The $\frk$-submodules are 

\begin{align*}
& \cspan  \{ 1 \} \cong   \cspan  \{ E_1 \wedge F_1 + E_2 \wedge F_2 \} \cong   \cspan   \{ E_1 \wedge E_2 \wedge F_1 \wedge F_2 \} \cong   V_{(0,0)} \\
& \cspan   \{ E_1, E_2 \} \cong   \cspan   \{ E_1 \wedge E_2 \wedge F_2, E_1 \wedge E_2 \wedge  F_1 \} \cong   V_{(1,0)} \\
& \cspan   \{ F_2, F_1 \} \cong   \cspan   \{ E_1 \wedge F_1 \wedge F_2, E_2 \wedge F_1 \wedge F_2 \} \cong   V_{(0, -1)} \\
& \cspan   \{ E_1 \wedge E_2 \} \cong   V_{(1,1)} \\
& \cspan   \{ F_1 \wedge F_2 \} \cong   V_{(-1, -1)} \\
& \cspan   \{ E_1 \wedge F_2, E_2 \wedge F_2 - E_1 \wedge F_1, E_2 \wedge F_1 \} \cong   V_{(1, -1)}.
\end{align*}
It follows that $\twedge(\frp)$ decomposes under $\frk$ as 
\begin{align*}
\twedge  (\p) & =  \overbrace{V_{(0, 0)}}^{4} \\
& \oplus  \overbrace{V_{(1,0)}}^{3} \oplus \overbrace{V_{(0,-1)}}^{3} \\
& \oplus  \overbrace{V_{(1,1)}}^{2} \oplus \overbrace{V_{(1,-1)}}^{2} \oplus \overbrace{V_{(-1,-1)}}^{2} \oplus \overbrace{V_{(0,0)}}^{2} \\
& \oplus  \overbrace{V_{(1,0)}}^{1} \oplus \overbrace{V_{(0,-1)}}^{1} \\
& \oplus  \overbrace{V_{(0,0)}}^{0},
\end{align*}
where each of the numbers over braces denotes the degree in which the corresponding $\frk$-module is appearing.

The element $a=H_1+H_2$ of $\frk$ acts on the module $V_{(\alpha,\beta)}$ by $\alpha+\beta$. It is clear that on any $K$-invariant, i.e., on any trivial $\frk$-module, $a$ has to act by 0. It does act by 0 on $\caH_{\frk_s}$. However, in $\caH_{\frk_s}\otimes\caH_{\frp}\otimes\twedge(\frp)$, $a$ will only act by 0 on tensor products where:
\begin{itemize}
\item $V_{(1, -1)}$ and $V_{(0,0)}$ in $\twedge(\p)$ are tensored with $V_{(k, -k)}$ in $\mathcal{H}_{\p}$; 
\item $V_{(1, 1)}$ in $\twedge(\p)$ is tensored with $V_{(k - 2, -k)}$ in $\mathcal{H}_{\p}$, $k \geq 2$; 
\item $V_{(-1, -1)}$ in $\twedge(\p)$ is tensored with $V_{(k+2, -k)}$ in $\mathcal{H}_{\p}$; 
\item $V_{(1,0)}$  in $\twedge(\p)$ is tensored with $V_{(k-1, -k)}$ in $\mathcal{H}_{\p}$, $k \geq 1$;
\item $V_{(0, -1)}$ in $\twedge(\p)$ is tensored with $V_{(k+1, -k)}$ in $\mathcal{H}_{\p}$.
\end{itemize}
In other words, all $K$--invariants in $\mathcal{H}_{\ka_s} \otimes \mathcal{H}_{\p} \otimes \twedge(\p)$ are contained in
\begin{align*}
\mathcal{H}_{\ka_s} \otimes \bigg( & \bigoplus_{k = 0}^{+ \infty} \overbrace{V_{(k, -k)}}^{2k} \otimes [\overbrace{V_{(1, -1)}}^{2} \oplus\overbrace{V_{(0, 0)}}^{4} \oplus \overbrace{V_{(0, 0)}}^{2} \oplus \overbrace{V_{(0, 0)}}^{0} ]  \\
\oplus & \bigoplus_{k = 2}^{+ \infty} \overbrace{V_{(k - 2, -k)}}^{2k - 2} \otimes \overbrace{V_{(1, 1)}}^{2} \oplus \bigoplus_{k = 0}^{+ \infty} \overbrace{V_{(k + 2, -k)}}^{2k+2} \otimes \overbrace{V_{(-1, -1)}}^{2} \\
\oplus & \bigoplus_{k = 1}^{+ \infty} \overbrace{V_{(k-1, -k)}}^{2k - 1} \otimes [\overbrace{V_{(1, 0)}}^{1} \oplus \overbrace{V_{(1, 0)}}^{3}] \oplus \bigoplus_{k = 0}^{+ \infty} \overbrace{V_{(k + 1, -k)}}^{2k + 1} \otimes [\overbrace{V_{(0, -1)}}^{1} \oplus \overbrace{V_{(0, -1)}}^{3}]\bigg).
\end{align*}

Since $a$ acts by zero on the above $\ka$--module, it is enough to regard it as an $\mathfrak{sl}(2, \mathbb{C})$--module:
\begin{align*}
\left( \bigoplus_{n = 0}^{+ \infty} \overbrace{V_{2n}}^{n} \right) \otimes \bigg( & \bigoplus_{k = 0}^{+ \infty} \overbrace{V_{2k}}^{2k} \otimes [ \overbrace{V_{2}}^{2} \oplus \overbrace{V_{0}}^{4} \oplus \overbrace{V_{0}}^{2} \oplus \overbrace{V_{0}}^{0} ]  \\
\oplus & \bigoplus_{k = 2}^{+ \infty} \overbrace{V_{2k - 2}}^{2k - 2} \otimes \overbrace{V_{0}}^{2} \oplus \bigoplus_{k = 0}^{+ \infty} \overbrace{V_{2k + 2}}^{2k+2} \otimes \overbrace{V_{0}}^{2} \\
\oplus & \bigoplus_{k = 1}^{+ \infty} \overbrace{V_{2k-1}}^{2k-1} \otimes [\overbrace{V_{1}}^{1} \oplus \overbrace{V_{1}}^{3}] \oplus \bigoplus_{k = 0}^{+ \infty} \overbrace{V_{2k + 1}}^{2k+1} \otimes [\overbrace{V_{1}}^{1} \oplus \overbrace{V_{1}}^{3}]
\bigg).
\end{align*}
Now using
\begin{align*}
& V_{2k} \otimes V_2 \cong   V_{2k+2} \oplus V_{2k} \oplus V_{2k-2} & \text{ for } k \geq 1 \\
& V_{2k} \otimes V_0 \cong   V_{2k}  & \text{for } k \geq 0 \\
& V_{2k-1} \otimes V_1 \cong   V_{2k}  \oplus V_{2k-2} & \text{ for } k \geq 1
\end{align*}
we have
\begin{align*}
\left( \bigoplus_{n = 0}^{+ \infty} \overbrace{V_{2n}}^{n} \right) & \otimes \bigg( \overbrace{V_2}^{2} \oplus \bigoplus_{k = 1}^{+ \infty} (\overbrace{V_{2k+2}}^{2k+2} \oplus \overbrace{V_{2k}}^{2k+2} \oplus \overbrace{V_{2k-2}}^{2k+2} ) \oplus \bigoplus_{k = 0}^{+ \infty} (\overbrace{V_{2k}}^{2k+4} \oplus \overbrace{V_{2k}}^{2k+2} \oplus \overbrace{V_{2k}}^{2k} ) \\
& \oplus ( \bigoplus_{k = 2}^{+ \infty} \overbrace{V_{2k-2}}^{2k}  ) \oplus ( \bigoplus_{k = 0}^{+ \infty} \overbrace{V_{2k+2}}^{2k + 4} ) \oplus \bigoplus_{k = 1}^{+ \infty} (\overbrace{V_{2k}}^{2k}  \oplus \overbrace{V_{2k-2}}^{2k} \oplus \overbrace{V_{2k}}^{2k+2} \oplus \overbrace{V_{2k-2}}^{2k+2} ) \\
& \oplus \bigoplus_{k = 0}^{+ \infty} (\overbrace{V_{2k+2}}^{2k+2}  \oplus \overbrace{V_{2k}}^{2k+2} \oplus \overbrace{V_{2k+2}}^{2k+4} \oplus \overbrace{V_{2k}}^{2k+4} ) \bigg).
\end{align*}
For $i \geq j$ we have $V_i \otimes V_j = V_{i+j} \oplus V_{i+j-2}\oplus\cdots \oplus V_{i-j}$. Therefore, an invariant (exactly one, up to scalar) will show up in $V_{i} \otimes V_{j}$ if and only if $i = j$.
It follows that the degrees of the invariants in the above tensor products, listed in the order of the summands of the second factor, are:
\begin{align*}
& 3, \quad 3k+3, 3k+2, 3k+1, \text{ for } k \geq 1, \quad 3k+4, 3k+2, 3k, \text{ for } k \geq 0, \\ & 3k-1, \text{ for } k \geq 2, \quad 3k+5,  \text{ for } k \geq 0, \quad 3k, 3k - 1, 3k + 2, 3k + 1, \text{ for } k \geq 1, \\
& 3k+3, 3k+2, 3k+5, 3k+4, \text{ for } k \geq 0.
\end{align*}

From this we conclude that in $\mathcal{H}_{\ka_s} \otimes \mathcal{H}_{\p} \otimes \twedge(\p)$ we have the following table

\bigskip

\begin{center}
\begin{tabular}{|c|c|}
  \hline
  degree & number of linearly independent invariants \\ \hline
  $0$ & $1$ \\
  $1$ & $0$ \\
  $2$ & $3$ \\
  $3k, \quad k \geq 1$ & $4$ \\
  $3k+1, \quad k \geq 1$ & $4$ \\
  $3k+2, \quad k \geq 1$ & $8$ \\
  \hline
\end{tabular}
\end{center}

\bigskip
\bigskip

In view of (\ref{sglpdec}), we have proved:

\begin{prop}
\label{propdeginv}
The algebra of $K$-invariants in $S(\frg)\otimes\twedge(\frp)$ can be written as
\[
(S(\frg)\otimes\twedge(\frp))^K=\bbC[a,b,c]\otimes (\mathcal{H}_{\ka_s} \otimes \mathcal{H}_{\p} \otimes \twedge(\p))^K.
\]
The number of invariants in the second factor in each degree is given by the above table.
\end{prop}

\section{A basis of $(S(\g) \otimes \twedge(\p))^{K}$}

Recall that we have defined elements $a,b$ and $c$ of $S(\frg)$. Now we view them as elements of $S(\g) \otimes \twedge(\p)$ by the identification
$S(\frg)=S(\frg)\otimes 1$. We also define further elements, which are all easily checked to be $K$-invariant:
\begin{align*}
& a = (H_1 + H_2) \otimes 1, \\
& b = (H^2 + 4 EF) \otimes 1, \\
& c = (E_1 F_1 + E_2 F_2) \otimes 1, \\
& d = (2 E E_2 F_1 + H E_1 F_1 - H E_2 F_2 + 2 F E_1 F_2) \otimes 1, \\
& e = F_1 \otimes E_1 + F_2 \otimes E_2, \\
& f = E_1 \otimes F_1 + E_2 \otimes F_2, \\
& g = 1 \otimes (E_1 \wedge F_1 + E_2 \wedge F_2), \\
& h = (2 E E_2 + H E_1) \otimes F_1 + (- H E_2 + 2 F E_1) \otimes F_2, \\
& i = 2 E \otimes E_2 \wedge F_1 + H \otimes E_1 \wedge F_1 - H \otimes E_2 \wedge F_2 + 2 F \otimes E_1 \wedge F_2, \\
& j = (H F_1 + 2 F F_2) \otimes E_1 + (2 E F_1 - H F_2) \otimes E_2.
\end{align*}
\begin{prop}\label{basis}
Let $S$ and $T$ be the following subsets of $(S(\g) \otimes \twedge(\p))^{K}$:
\begin{align*}
& S = \{ a^{n_1} b^{n_2} c^{n_3} d^{n_4} \, | \, n_1, n_2, n_3, n_4 \in \mathbb{N}_{0} \} \\
& T = \{ 1, e, f, g, h, i, j, ef, eg, fg, g^2, ei, ej, fh, fi, fj \}.
\end{align*}
Then the set $S \cdot T$ of products of elements of $S$ and $T$ in the algebra $S(\g) \otimes \twedge(\p)$ is a
basis for $(S(\g) \otimes \twedge(\p))^{K}$.
\end{prop}
\begin{proof}  
We first prove the linear independence of the set $S \cdot T$. Notice that it is enough to prove the linear independence of following sets:
\begin{enumerate}[a)]
\item $S$,
\item $S \cdot \{ e \} \cup S \cdot \{ j \}$,
\item $S \cdot \{ f \} \cup S \cdot \{ h \}$,
\item $S \cdot \{ i \} \cup S \cdot \{ ef \} \cup S \cdot \{ fj \} \cup S \cdot \{ g \}$,
\item $S \cdot \{ eg \} \cup S \cdot \{ ei \}$,
\item $S \cdot \{ fg \} \cup S \cdot \{ fi \}$.
\end{enumerate}
Namely, the rest of the independence then follows by considering just the second factors in the tensor products. 
We deal with each of the cases a) -- f) separately.

\begin{enumerate}[a)]
\item Let $\sum_{i \in \mathcal{I}} \lambda_{i} \cdot a^{n_{1,i}} b^{n_{2,i}} c^{n_{3, i}} d^{n_{4, i}} = 0$ and let $\{ m_1, m_2, \cdots, m_k \} = \{ n_{1, i} \, | \, i \in \mathcal{I} \}$, where $m_1 < m_2 < \cdots < m_k$. Since one basis of $\g$ is given by $\{ H_1 + H_2, H, E, F, E_1, F_1, E_2, F_2 \}$, from
$$
\sum_{j = 1}^{k} a^{m_j} \left(\sum_{i \in \mathcal{I}, n_{1, i} = m_j} \lambda_{i} b^{n_{2, i}} c^{n_{3, i}} d^{n_{4, i}} \right) = 0,
$$
we have $\sum_{i \in \mathcal{I}, n_{1, i} = m_j} \lambda_{i} b^{n_{2, i}} c^{n_{3, i}} d^{n_{4, i}} = 0$ for all $j \in \{1, \cdots, k\}$. It is thus enough to prove that the set $\{ b^{n_2} c^{n_3} d^{n_4}\, | \, n_2, n_3, n_4 \in \mathbb{N}_{0}\}$ is linearly independent. Let $\sum_{i \in \mathcal{I}}  \lambda_{i} b^{n_{2, i}} c^{n_{3, i}} d^{n_{4, i}} = 0$. In the expansion of the summand
$$
(H^2 + 4 EF)^{n_{2, i}}(E_1 F_1 + E_2 F_2)^{n_{3, i}}(2 E E_2 F_1 + H(E_1 F_1 - E_2 F_2) + 2 F E_1 F_2)^{n_{4, i}}
$$ we consider the terms without $H$ and $F_2$. There is only one such term and it is $(4E F)^{n_{2, i}}(E_1 F_1)^{n_{3, i}}(2 E E_2 F_1)^{n_{4, i}}$. We have
$$
\sum_{i \in \mathcal{I}} \lambda_{i} (4EF)^{n_{2, i}}(E_1 F_1)^{n_{3, i}} (2E E_2 F_1)^{n_{4, i}} = 0.
$$
Let $\{t_1, \cdots, t_l \} = \{ n_{4, i} \, | \, i \in \mathcal{I} \}, t_1 < t_2 < \cdots < t_l$. We have
$$
\sum_{j = 1}^{l} (2E_{2})^{t_j} (\sum_{i \in \mathcal{I}, n_{4, i} = t_j} \lambda_i (4EF)^{n_{2, i}}(E_1 F_1)^{n_{3, i}}(EF_1)^{t_j}) = 0
$$
and from here $\sum_{i \in \mathcal{I}, n_{4, i} = t_j} \lambda_i (4EF)^{n_{2, i}}(E_1 F_1)^{n_{3, i}} = 0$ for all $j \in \{ 1, \cdots, l \}$.
It follows $\lambda_i = 0$ for all $j \in \{ 1, \cdots, l \}$.

\item We consider summands of the form $\cdot \otimes E_1$. As in the previous case, it is enough to show the linear independence of the set
$$
\{ b^{n_2} c^{n_3} d^{n_4} F_1 \, | \, n_2, n_3, n_4 \in \mathbb{N}_{0} \} \cup \{ b^{n_2} c^{n_3} d^{n_4}(HF_1 + 2 F F_2) \, | \, n_2, n_3, n_4 \in \mathbb{N}_{0} \}.
$$
Let $\sum_{i \in \mathcal{I}} \lambda_i b^{n_{2, i}} c^{n_{3, i}} d^{n_{4, i}} F_1 + \sum_{j \in \mathcal{J}} \lambda_j b^{n_{2, j}} c^{n_{3, j}} d^{n_{4, j}} (H F_1 + 2 F F_2) = 0$.
As in case $a)$, we consider the summands without $H$ and $F_2$ and we get
$$
\sum_{i \in \mathcal{I}} \lambda_i (4EF)^{n_{2, i}} (E_1 F_1)^{n_{3, i}} (2E E_2 F_1)^{n_{4, i}} F_1 = 0.
$$
By the same arguments as in case $a)$, we get $\lambda_i = 0$ for all $i \in \mathcal{I}$. Then we have $\sum_{j \in \mathcal{J}}\lambda_j b^{n_{2, j}} c^{n_{3, j}} d^{n_{4, j}} (H F_1 + 2 F F_2) = 0$ and $$\sum_{j \in \mathcal{J}} \lambda_j b^{n_{2, j}} c^{n_{3, j}} d^{n_{4, j}} = 0.$$ Since we have already proved that the set $\{ b^{n_2} c^{n_3} d^{n_4}\, | \, n_2, n_3, n_4 \in \mathbb{N}_{0}\}$ is linearly independent, it follows $\lambda_j = 0$ for all $j \in \mathcal{J}$.

\item Let us consider summands of the form $\cdot \otimes F_1$. It is enough to prove the linear independence of the set
$$
\{ b^{n_2} c^{n_3} d^{n_4} E_1 \, | \, n_2, n_3, n_4 \in \mathbb{N}_{0} \} \cup \{ b^{n_2} c^{n_3} d^{n_4}(HE_1 + 2 E E_2) \, | \, n_2, n_3, n_4 \in \mathbb{N}_{0} \}.
$$
The proof is similar to the case $b)$, only in this case we consider summands without $H$ and $E_2$.

\item In this case we consider summands of the form $\cdot \otimes E_2 \wedge F_2$. It is enough to prove the linear independence of the set
\begin{align*}
& \{ b^{n_2} c^{n_3} d^{n_4} \, | \, n_2, n_3, n_4 \in \mathbb{N}_{0} \} \cup \{ b^{n_2} c^{n_3} d^{n_4} (H E_2 F_2 - 2 E E_2 F_1) \, | \, n_2, n_3, n_4 \in \mathbb{N}_{0} \} \\
& \cup \{ b^{n_2} c^{n_3} d^{n_4} E_2 F_2 \, | \, n_2, n_3, n_4 \in \mathbb{N}_{0} \} \cup \{ b^{n_2} c^{n_3} d^{n_4} H \, | \, n_2, n_3, n_4 \in \mathbb{N}_{0} \}.
\end{align*}
Let 
\begin{align*}
& \sum_{i \in \mathcal{I}} \lambda_i b^{n_{2, i}} c^{n_{3, i}} d^{n_{4, i}} +  \sum_{j \in \mathcal{J}} \lambda_j b^{n_{2, j}} c^{n_{3, j}} d^{n_{4, j}} H \\
& + \sum_{k \in \mathcal{K}} \lambda_k b^{n_{2, k}} c^{n_{3, k}} d^{n_{4, k}} E_2 F_2 +  \sum_{l \in \mathcal{L}} \lambda_l b^{n_{2, l}} c^{n_{3, l}} d^{n_{4, l}} (HE_2 F_2 - 2E E_2 F_1) = 0.
\end{align*}
We first consider the summands without $H$ and $E_2$ and get
$$
\sum_{i \in \mathcal{I}} \lambda_i (4EF)^{n_{2, i}}(E_1 F_1)^{n_{3, i}}(2F E_1 F_2)^{n_{4, i}} = 0.
$$
Then it follows $\lambda_i = 0$ for all $i \in \mathcal{I}$, similarly as in case $a)$. Now we consider summands without $H$ and $F_1$ and get
$$
\sum_{k \in \mathcal{K}} \lambda_k (4EF)^{n_{2, k}}(E_2 F_2)^{n_{3, k}}(2F E_1 F_2)^{n_{4, k}} E_2 F_2 = 0.
$$
It follows $\lambda_k = 0$ for all $k \in \mathcal{K}$. Then we consider summands without $H$ and $F_2$ and similarly as before get $\lambda_l = 0$ for all $l \in \mathcal{L}$. Finally, we have
$$
\sum_{j \in \mathcal{J}} \lambda_j b^{n_{2, j}}c^{n_{3, j}}d^{n_{4, j}} H = 0.
$$
We conclude $\lambda_j = 0$ for all $j \in \mathcal{J}$.

\item We consider summands of the form $\cdot \otimes E_1 \wedge E_2 \wedge F_1$. Let
$$
\sum_{i \in \mathcal{I}} \lambda_i b^{n_{2, i}} c^{n_{3, i}} d^{n_{4, i}} F_2 + \sum_{j \in \mathcal{J}} \lambda_j b^{n_{2, j}} c^{n_{3, j}} d^{n_{4, j}}(2 E F_1 - H F_2) = 0.
$$
By considering summands without $H$ and $F_2$ we get $\lambda_j = 0$ for all $j \in \mathcal{J}$ and then, from the independence of the set $\{ b^{n_2} c^{n_3} d^{n_4} \, | \, n_2, n_3, n_4 \in \mathbb{N}_{0} \}$ it follows $\lambda_i = 0$ for all $i \in \mathcal{I}$.

\item This time we consider summands of the form $\cdot \otimes E_1 \wedge F_1 \wedge F_2$ and then those without $H$ and $E_2$. The proof is similar as in the previous case.
\end{enumerate}

This finishes the proof of linear independence of the set $S\cdot T$. To prove that $S\cdot T$ is also a spanning set, we use Proposition \ref{propdeginv}. We consider the degrees of the elements of the set $\{d^{n_4}\}\cdot T$, i.e., of
\begin{align*}
& \{ d^{n_4}, d^{n_4}e, d^{n_4}f, d^{n_4}g, d^{n_4}h, d^{n_4}i, d^{n_4}j, \\
& d^{n_4} ef, d^{n_4}eg, d^{n_4}fg, d^{n_4}g^2, d^{n_4}ei, d^{n_4} ej, d^{n_4}fh, d^{n_4}fi, d^{n_4}fj \}.
\end{align*}
The degrees of these elements are respectively 
\begin{align*}
& 3 n_4, 3 n_4 + 2, 3 n_4 + 2, 3 n_4 + 2, 3 n_4 + 3, 3 n_4 + 3, 3 n_4 + 3, \\
& 3 n_4 + 4, 3 n_4 + 4, 3 n_4 + 4, 3 n_4 + 4, 3 n_4 + 5, 3 n_4 + 5, 3 n_4 + 5, 3 n_4 + 5, 3 n_4 + 5.
\end{align*}
Varying $n_4$ and considering the number of invariants in each degree of the set $\{d^{n_4}\,\big|\, n_4\in\bbZ_+\}\cdot T$, we get the following table:

\bigskip

\begin{center}
\begin{tabular}{|c|c|}
  \hline
  degree & number of invariants \\ \hline
  $0$ & $1$ \\
  $1$ & $0$ \\
  $2$ & $3$ \\
  $3k, \quad k \geq 1$ & $4$ \\
  $3k+1, \quad k \geq 1$ & $4$ \\
  $3k+2, \quad k \geq 1$ & $8$ \\
  \hline
\end{tabular}
\end{center}

\bigskip

In view of Proposition \ref{propdeginv}, this finishes the proof.
\end{proof}  

The above proposition shows in particular that the algebra $(S(\frg)\otimes \twedge(\frp))^K$ is generated by elements $a,\dots,j$. It also proves the following facts:

\begin{cor}
\label{corSg}
The algebra $S(\frg)^K$ is a polynomial algebra generated by the elements $a,b,c,d$. The algebra $(S(\frg)\otimes \twedge(\frp))^K$ is a free module over $S(\frg)^K$ of rank 16, and the elements of the set $T$ form a basis for this free module.
\end{cor}

\section{The set of generators for $(U(\g) \otimes C(\p))^{K}$}

In this section we will show that the algebra of $K$--invariants in $U(\g) \otimes C(\p)$ is generated by five explicit elements. Two of these elements generate $Z(\frk)$ and the third generates the algebra $C(\frp)^K$. The fourth element is the Dirac operator, and the fifth can be considered as a
$\frk$-analogue of the Dirac operator.

Recall that the symmetrization map $\sigma : S(\g) \longrightarrow U(\g)$ given by
$$
\sigma(x_1 x_2 \cdots x_n)= \frac{1}{n!} \sum_{\alpha \in S_n} x_{\alpha(1)} x_{\alpha(2)} \cdots x_{\alpha(n)}, \quad n \in \mathbb{N}, x_1, \cdots, x_n \in \g
$$
is an isomorphism of $K$--modules. The Chevalley map $\tau: \twedge(\p) \longrightarrow C(\p)$ given by the composition of the map
$$
v_1 \wedge \cdots \wedge v_n \mapsto \frac{1}{n!} \sum_{\alpha \in S_n} \text{sgn}(\alpha) v_{\alpha(1)} \otimes \cdots \otimes v_{\alpha(n)}
$$
from $\twedge(\p)$ into the tensor algebra $T(\p)$ and the canonical projection from $T(\p)$ to $C(\p)$ is also an isomorphism of $K$--modules.

Since for $z_1, \cdots, z_n \in \g$ and $\alpha \in S_n$ we have $z_1 \cdots z_n - z_{\alpha(1)} \cdots z_{\alpha(n)} \in U_{n-1}(\g)$, it follows
\begin{equation}\label{universal}
\sigma(z_1 \cdots z_n) = z_1 \cdots z_n \quad\text{modulo } U_{n-1}(\g).
\end{equation}
Similarly, since for $y_1, \cdots, y_k \in \p$ and $\alpha \in S_k$ we have $y_1 \cdots y_k - \text{sgn}(\alpha) y_{\alpha(1)} \cdots y_{\alpha(k)} \in C_{k-1}(\p)$, it follows
\begin{equation}\label{clifford}
\tau(y_1 \wedge \cdots \wedge y_k) = y_1 \cdots y_k \quad\text{modulo }  C_{k-1}(\p). 
\end{equation}
Consider the following elements of $U(\g) \otimes C(\p)$:  
\begin{align*}
& \tilde{a} = (H_1 + H_2) \otimes 1, \\
& \tilde{b} = (H^2 + 2(EF + FE)) \otimes 1, \\
& \tilde{c} = (E_1 F_1 + E_2 F_2) \otimes 1, \\
& \tilde{d} = (2 E E_2 F_1 + H E_1 F_1 - H E_2 F_2 + 2 F E_1 F_2) \otimes 1, \\
& \tilde{e} = F_1 \otimes E_1 + F_2 \otimes E_2, \\
& \tilde{f} = E_1 \otimes F_1 + E_2 \otimes F_2, \\
& \tilde{g} = 1 \otimes (E_1  F_1 + E_2  F_2), \\
& \tilde{h} = (2 E E_2 + H E_1) \otimes F_1 + (- H E_2 + 2 F E_1) \otimes F_2, \\
& \tilde{i} = 2 E \otimes E_2  F_1 + H \otimes E_1  F_1 - H \otimes E_2 F_2 + 2 F \otimes E_1  F_2, \\
& \tilde{j} = (H F_1 + 2 F F_2) \otimes E_1 + (2 E F_1 - H F_2) \otimes E_2.
\end{align*}
One can check that
\begin{align}\label{sigmatau}
& (\sigma \otimes \tau) (a) = \tilde{a}, \\
& (\sigma \otimes \tau) (b) = \tilde{b}, \notag \\
& (\sigma \otimes \tau) (c) = \tilde{c} - \frac{3}{2} \tilde{a}, \notag \\
& (\sigma \otimes \tau) (d) = \tilde{d} - \frac{1}{2}\tilde{b} - \frac{3}{2} \tilde{a}, \notag \\
& (\sigma \otimes \tau) (e) = \tilde{e}, \notag \\
& (\sigma \otimes \tau) (f) = \tilde{f}, \notag \\
& (\sigma \otimes \tau) (g) = \tilde{g} + 2 \otimes 1, \notag \\
& (\sigma \otimes \tau) (h) = \tilde{h} - \frac{3}{2} \tilde{f}, \notag \\
& (\sigma \otimes \tau) (i) = \tilde{i}, \notag \\
& (\sigma \otimes \tau) (j) = \tilde{j} + \frac{3}{2} \tilde{e} \notag 
\end{align}

Using this, \eqref{universal}, \eqref{clifford} and Proposition \ref{basis} one shows by induction that the following lemma holds:

\begin{lemma}
\label{ugcpk10gen}
The algebra $(U(\g) \otimes C(\p))^{K}$ is generated by elements $\tilde{a}$, $\tilde{b}$, $\tilde{c}$, $\tilde{d}$, $\tilde{e}$, $\tilde{f}$, $\tilde{g}$, $\tilde{h}$, $\tilde{i}$ and $\tilde{j}$. 
\end{lemma}

Let 
\[
D = E_1 \otimes F_1 + E_2 \otimes F_2 + F_1 \otimes E_1 + F_2 \otimes E_2
\] 
be the Dirac operator. 
Using $D$, we can reduce the set of generators for $(U(\g) \otimes C(\p))^{K}$ given by Lemma \ref{ugcpk10gen}, since we have
\begin{align*}
\tilde{e} & = \frac{1}{2}(D + \frac{1}{2} D \tilde{g} - \frac{1}{2} \tilde{g} D), \\
\tilde{f} & = \frac{1}{2}(D - \frac{1}{2} D \tilde{g} + \frac{1}{2} \tilde{g} D), \\
\tilde{c} & = - \frac{1}{4} (\tilde{i} + 2(\tilde{e} \tilde{f} + \tilde{f} \tilde{e}) + 3\tilde{a} \tilde{g}), \\
\tilde{h} & = 2(\tilde{f} \tilde{c} - \tilde{c} \tilde{f}) - 3\tilde{a} \tilde{f} + 6 \tilde{f}, \\
\tilde{j} & = 2(\tilde{c} \tilde{e} - \tilde{e} \tilde{c}) - 3\tilde{a} \tilde{e}, \\
\tilde{d} & = \frac{1}{2}(- \tilde{h} \tilde{e} - \tilde{e} \tilde{h} + 2 \tilde{c} \tilde{g} - \tilde{e} \tilde{f} - 6 \tilde{a} \tilde{g} - \frac{1}{2} \tilde{b} \tilde{g} - \tilde{i} - \frac{3}{2} \tilde{a} \tilde{i}). \\
\end{align*}
From this we conclude that the algebra $(U(\g) \otimes C(\p))^{K}$ is generated by elements $\tilde{a}, \tilde{b}, D, \tilde{g}, \tilde{i}$. 

The element $\tilde{a}$ is in the center of $\ka$. The element $\tilde{b}$ is up to scalar equal to the Casimir element of $\frk_s=\mathfrak{sl}(2,\mathbb{C})$. Thus $\tilde{a}$ and $\tilde{b}$ generate $Z(\frk)$. The element $\tilde{g}$ generates the three-dimensional algebra $C(\frp)^K$.

Furthermore, note that the basis $\left(E, F, H_1 - H_2, H_1 + H_2 \right)$ of $\ka$ is dual for the basis 
\[
\left(F, E, \frac{1}{2}(H_1 - H_2), \frac{3}{2}(H_1 + H_2) \right)
\]
with respect to the trace form. 
Let $\alpha : \frk \longrightarrow C(\p)$ be the action map $\ka \longrightarrow \mathfrak{so}(\p)$ composed with the inclusion 
$$
\mathfrak{so}(\p) \cong   \twedge^{2}(\p) \hookrightarrow C(\p)
$$
(see \cite{HP2}). Then the element
\begin{multline*}
E \otimes \alpha(F) + \frac{1}{2}(H_1 - H_2) \otimes \alpha(H_1 - H_2) +\\ 
\qquad\qquad\frac{3}{2}(H_1 + H_2) \otimes \alpha(H_1 + H_2) + F \otimes \alpha(E) = \\
 - \frac{1}{4}(2 E \otimes E_2 F_1 + (H_1 - H_2) \otimes (E_1 F_1 - E_2 F_2) + \qquad\qquad\qquad\qquad\qquad\qquad\\ 
 3(H_1 + H_2) \otimes (E_1 F_1 + E_2 F_2) + 2 F \otimes E_1 F_2) - \frac{3}{2}(H_1 + H_2) \otimes 1 =\\
  -\frac{1}{4}(\tilde{i} + 3\tilde{a} \tilde{g}) - \frac{3}{2} \tilde{a}
\end{multline*}
can be thought of as a $\ka$-version of the Dirac operator, and we denote it by $D^\frk$. 
It is clear that we can replace $\tilde i$ by this element and still get a set of generators.

We have proved:

\begin{theorem}
The algebra $(U(\g) \otimes C(\p))^{K}$ is generated by the following five elements:
\begin{itemize}
\item $\tilde a$ and $\tilde b$, generating $Z(\frk)$;
\item $\tilde g$, generating $C(\frp)^K$;
\item The Dirac operator $D$ and its $\frk$-version $D^\frk$.
\end{itemize}
\end{theorem}

As a consequence, we get the following corollary about the quotient of the algebra $(U(\g) \otimes C(\p))^{K}$ by the ideal generated by $D$. This quotient algebra is important because it acts on Dirac cohomology of any $(\frg,K)$-module.

\begin{cor}
Let $\mathcal{I}$ be the ideal in the algebra $(U(\g) \otimes C(\p))^{K}$ generated by the Dirac operator $D$. The quotient algebra $(U(\g) \otimes C(\p))^{K}/\mathcal{I}$ is generated by classes of elements $\tilde{a}$, $\tilde{b}$, $\tilde{i}$ and $\tilde{g}$ and it is abelian. 
\end{cor}

\begin{proof}  
The algebra generated by elements $\tilde{a}$, $\tilde{b}$, $\tilde{i}$ and $\tilde{g}$ is abelian and it is a subset of the algebra $U(\ka) \otimes C(\p)$. Since $D$ has an element of $\frp$ in the first factor of each summand, while elements of $U(\ka) \otimes C(\p)$ have no elements of $\frp$ in the first factors of any of their summands, we can use the Poincar\'{e}-Birkhoff-Witt theorem to see
$$
(U(\ka) \otimes C(\p)) \cap \mathcal{I} = \{ 0 \}.
$$
The claim follows. 
\end{proof}  

By results of \cite{PR}, Section 4, the commutativity of the algebra $(U(\g) \otimes C(\p))^{K}/\mathcal{I}$ from the above corollary leads to the fact that for any irreducible $(\frg,K)$-module $X$, its Dirac cohomology $H_D(X)$ has $\Kt$-multiplicities equal to 1. By results of \cite{BP2}, this is known to be false for general $SU(p,q)$, so $(U(\g) \otimes C(\p))^{K}/\mathcal{I}$ will not be abelian in general.

The following result is a special case of a more general result for the Lie algebra $\g = \mathfrak{su}(n,1)$ \cite{J}.

\begin{cor}
\label{corugk}
The subalgebra $U(\g)^{K}$ of $U(\frg)$ is equal to $Z(\g) Z(\ka)$.
\end{cor}

\begin{proof} 
From the theorem of Chevalley \cite{C} it follows that $Z(\g)$ is a polynomial algebra on two generators of degrees $2$ and $3$.
For these generators, we can take the Casimir element of $U(\g)$, 
\begin{multline*}
\Omega = \frac{1}{2} (H_1 - H_2)^2 + \frac{3}{2} (H_1 + H_2)^{2} + EF + FE + E_1 F_1 + E_2 F_2 + F_1 E_1 + F_2 E_2 =\\ 
\frac{3}{2} \tilde{a}^2 + \frac{\tilde{b}}{2} + 2\tilde{c} - 3\tilde{a},
\end{multline*}
and the element
\[
\text{cub} = - \frac{3}{2} \tilde{a}^{3} + \frac{3}{2} \tilde{a} \tilde{b} - 3 \tilde{a} \tilde{c} + \frac{9}{2} \tilde{a}^2 - 3 \tilde{a} + 3 \tilde{d} - \frac{3}{2} \tilde{b}.
\]
Since the algebra $U(\g)^{K}$ is generated by the elements $\tilde{a}, \tilde{b}, \tilde{c}, \tilde{d}$ and since $\tilde{a}$ and $\tilde{b}$ generate $Z(\ka)$, the claim follows. 
\end{proof}

Finally, we can obtain the following description of the algebra $(U(\g) \otimes C(\p))^K$. 
Let $\tilde{S}$ and $\tilde{T}$ be the following subsets of $(U(\g) \otimes C(\p))^K$:
\begin{align*}
\tilde{S} & = \{ \tilde{a}^{n_1} \tilde{b}^{n_2} \tilde{c}^{n_3} \tilde{d}^{n_4} \, | \, n_1, n_2, n_3, n_4 \in \mathbb{N}_{0} \}
\\
\tilde{T} & = \{ 1, \tilde{e}, \tilde{f}, \tilde{g}, \tilde{h}, \tilde{i}, \tilde{j}, \tilde{e}\tilde{f}, \tilde{e}\tilde{g}, \tilde{f}\tilde{g}, \tilde{g}^2, \tilde{e}\tilde{i}, \tilde{e}\tilde{j}, \tilde{f}\tilde{h}, \tilde{f}\tilde{i}, \tilde{f}\tilde{j} \} \notag
\end{align*}
\begin{cor}
\label{corugcpk}
The algebra $(U(\g) \otimes C(\p))^K$ is a free module over $U(\g)^{K}$ of rank $16$, and the elements of the set $\tilde{T}$ form a basis for this free module. 
\end{cor}

\begin{proof} 
It is enough to prove that the set $\tilde{S} \cdot \tilde{T}$ of products of elements of $\tilde{S}$ and $\tilde{T}$ in the algebra $U(\g) \otimes C(\p)$  is a basis for $(U(\g) \otimes C(\p))^{K}$. Using \eqref{universal}, \eqref{clifford}, \eqref{sigmatau} and Proposition \ref{basis} one shows by induction that the set $\tilde{S} \cdot \tilde{T}$ spans $(U(\g) \otimes C(\p))^K$. Since for each positive integer $n$ we have 
\[
U_{n}(\g) = \sigma(S^{n}(\g)) \oplus U_{n-1}(\g)\qquad\text{and}\qquad C_{n}(\p) = \tau(\twedge^{n}(\p)) \oplus C_{n-1}(\p),
\] 
the following identity holds for all positive integers $n$ and $m$:
\begin{align*}
U_{n}(\g) \otimes C_{m}(\p) & = (\sigma \otimes \tau) (S^{n}(\g) \otimes \twedge^{m}(\p)) \oplus (U_{n-1}(\g) \otimes \tau(\twedge^{m}(\p)) \\
& \oplus (\sigma(S^{n}(\g)) \otimes C_{m-1}(\p)) \oplus (U_{n-1}(\g) \otimes C_{m-1}(\p))).
\end{align*}
Linear independence now follows easily by induction on $n$ and $m$ from \eqref{sigmatau} and Proposition \ref{basis}.
\end{proof}

\end{document}